\documentclass[12pt]{amsart}

\usepackage{a4,color}

\usepackage{amsmath, amssymb}

\usepackage{mathrsfs}

\usepackage{exscale,txfonts}

\usepackage{enumerate}

\renewcommand{\iff}{\Leftrightarrow}

\newcommand{\ox}{\otimes}
\newcommand{\got}[1]{\mathfrak{#1}}

\renewcommand{\bar}{\overline}

\newcommand{\bbar}{\bar{\phantom{x}}}

\newcommand{\axstar}{\langle\bar X,\bar X^*\rangle}
\newcommand{\T}{t}

\DeclareMathOperator{\GM}{GM}
\DeclareMathOperator{\UD}{UD}
 \DeclareMathOperator{\Int}{Int}
\DeclareMathOperator{\Sym}{Sym} 
 
\DeclareMathOperator{\Trd}{Trd} 
\DeclareMathOperator{\Nrd}{Nrd} 
 \DeclareMathOperator{\tr}{tr}

\DeclareMathOperator{\sign}{sign}

\newcommand{\rc}{ {\rm rc} }

\newcommand{\<}{\langle}
\renewcommand{\>}{\rangle}
\newcommand{\x}{\times}

\DeclareMathOperator{\ad}{ad}
\DeclareMathOperator{\J}{J}

\newcommand{\sos}[1]{\ensuremath{\sum{#1}^{2}}}

\newcommand{\R}{\mathbb{R}}
\newcommand{\C}{\mathbb{C}}
\newcommand{\N}{\mathbb{N}}

\newcommand{\KK}{\mathbb{K}}

\newcommand{\ve}{\varepsilon}
\newcommand{\vf}{\varphi}

\newcommand{\s}{\sigma}

\newcommand{\GL}{\mathrm{GL}}

\newtheorem{lemma}{Lemma}[section]
\newtheorem{theorem}[lemma]{Theorem}
\newtheorem{propo}[lemma]{Proposition}
\newtheorem{coro}[lemma]{Corollary}

\newtheorem*{conj}{The PS~Conjecture}

\theoremstyle{definition}
\newtheorem{defi}[lemma]{Definition}

\newtheorem{remark}[lemma]{Remark}
\newtheorem{exam}[lemma]{Example}

\title[The Procesi--Schacher conjecture]{The Procesi--Schacher conjecture and Hilbert's 17th problem for algebras with involution}
\author{Igor Klep}
\author{Thomas Unger}
\address{Univerza v Mariboru, Fakulteta za matematiko in naravoslovje, Koro\v ska 160, SI-2000 Maribor, Slovenia \\ and  
Univerza v Ljubljani, Fakulteta za matematiko in fiziko, Jadranska 19, SI-1111 Ljubljana, Slovenia}
\email{igor.klep@fmf.uni-lj.si}
\address{School of Mathematical Sciences, University College Dublin, Belfield, Dublin~4, Ireland}
\email{thomas.unger@ucd.ie}
\date{04 May 2009}
\begin{document}

\begin{abstract}
In 1976 Procesi and Schacher developed an Artin--Schreier 
type theory for central simple algebras with involution and
conjectured that in such an algebra a totally positive element
is always a sum of hermitian squares. 
In this paper elementary counterexamples to this conjecture are constructed and cases are studied where the conjecture does hold. Also, a Positivstellensatz is established for noncommutative polynomials,
positive semidefinite on all tuples of matrices of a fixed size.
\end{abstract}
\maketitle
\vspace{-.5cm}
\begin{center}
\small
\emph{Dedicated to David W. Lewis on the occasion of his 65th birthday.}
\end{center}

\section{Introduction}

Artin's 1927 affirmative solution of  Hilbert's 17th problem (\emph{Is every nonnegative real polynomial a sum of squares of rational functions?})
arguably sparked the beginning of the field of real algebra and consequently
real algebraic geometry (cf. \cite{BCR, PD}).

Starting with Helton's seminal paper \cite{Hel}, in which he proved
that every positive semidefinite real or complex noncommutative polynomial is a
sum of hermitian squares of \emph{polynomials}, variants of Hilbert's 
17th problem in a \emph{noncommutative} setting have become a topic of current
interest with wide-ranging applications 
(e.g. in control theory, optimization, engineering, mathematical 
physics, etc.); see \cite{dOHMP} for
a nice survey. Most of these results have a functional analytic
flavour and are what Helton et al.~call \emph{dimensionfree}, that is,
they deal with evaluations of noncommutative polynomials in
matrix algebras of arbitrarily large size.

Procesi and Schacher in their 
1976 Annals of Mathematics paper \cite{PS}
introduce a notion
of orderings on central simple algebras with involution, 
prove a real Nullstellensatz, and a weak noncommutative version of
Hilbert's 17th problem. A strengthening of the latter is proposed
as a conjecture \cite[p.~404]{PS}:
\emph{In a central simple algebra with involution,
a totally positive element
is always a sum of hermitian squares.}

We explain in Section~\ref{sec5} how these results can be applied
to study 
non-dimensionfree positivity of noncommutative polynomials.
Roughly speaking, a noncommutative polynomial all of whose evaluations
in $n\times n$ matrices (for \emph{fixed} $n$) are positive semidefinite,
is a sum of hermitian squares with denominators and weights.

A brief outline of the rest of the paper is as follows: in
Section~\ref{sec:PSc} we fix terminology and summarize some of the
Procesi--Schacher results in a modern
language. Then in Section~\ref{sec:counter} we present counterexamples
to the Procesi--Schacher conjecture, while Section~\ref{para.pos.rel}
contains a study of examples (mainly in the split case) where the conjecture
is true.

For general background on central simple algebras with involution we refer the reader to \cite{BOI} and for the theory of quadratic forms over fields we refer to \cite{Lam}.

\section{The Procesi--Schacher Conjecture}\label{sec:PSc}

Let $F$ be a formally real field and let $A$ be a central simple algebra with involution $\s$ and centre $K$. Assume that $F$ is the fixed field of $\s$ (i.e., $\s|_F=\mathrm{id}_F$). The involution $\s$ is of the \emph{first kind} if  $K=F$, and  of the \emph{second kind} (also called \emph{unitary}) otherwise. In this case $[K:F]=2$ and $\s|_K$ is the non-trivial element in $\mathrm{Gal}(K/F)$.

Let $D$ be a division algebra over $K$ with involution $\tau$ and fixed field $F$. Let $h$ be an $n$-dimensional hermitian or skew-hermitian form over $(D,\tau)$. Then $h$ gives rise to an involution on $M_n(D)$,  the \emph{adjoint involution} $\ad_h$, defined by
\[\ad_h(X)= H\cdot \tau(X)^t\cdot H^{-1},\] 
for all $X\in M_n(D)$, where $H$ is the Gram matrix of $h$, $t$ denotes the transpose map on $M_n(D)$ and $\tau(X)$ signifies applying $\tau$ to the entries of $X$. It is well-known that every central simple algebra with involution $(A,\s)$ is of the form $(M_n(D),\ad_h)$, where $n$ is unique, $D$ is unique up to isomorphism and $h$ is unique up to multiplicative equivalence (see \cite[4.A]{BOI}).

If $\s$ is of the first kind, then $\s$ is called \emph{orthogonal} or \emph{symplectic} if $\s$ becomes adjoint to a quadratic or alternating form, respectively, after scalar extension to a \emph{splitting field} of $A$ (i.e., an extension field $L$ of $K$ such that $A\ox_K L\cong M_n(L)$).
We denote the subspace of $\s$-symmetric elements of $A$ by $\Sym(A,\s)$.

Let $\leq$ be an ordering on $F$. We identify $\leq$ with its \emph{positive cone} $P=\{x\in F\,|\, 0\leq x\}$ via
\[x\leq y \iff y-x\in P\]
for all  $x,y\in F$.
In this case we also write $\leq_P$ instead of $\leq$.

Procesi and Schacher \cite[\S 1]{PS} consider  central simple algebras $A$, equipped with a \emph{positive involution} $\s$, i.e., an involution whose \emph{involution trace form}  $T_\s$ is positive semidefinite with respect to the ordering $\leq_P$ on $F$, 
\[T_\s(x):=\Trd(\sigma(x)x)\geq_P 0 \quad\text{for all } x\in A.\]
Here $\Trd:A\to F$ (the \emph{trace}) denotes the reduced trace $\Trd_{A/F}$ if $\s$ is of the first kind and the composition $\Trd_{K/F}\circ \Trd_{A/K}$ if $\s$ is of the second kind. 
The form $T_\s$ is a nonsingular quadratic form over $F$, cf. \cite[\S11]{BOI}.
If $\dim_K A=n$, then $\dim T_\s=n$ if $\s$ is of the first kind and $\dim T_\s=2n$ if $\s$ is of the second kind.

\begin{remark} The notion of positive involution seems to have been considered first by Weil in his groundbreaking paper \cite{Wei}. Lewis and Tignol \cite{LT} define the signature of an involution $\s$ of the first kind on $A$ with respect to the ordering $\leq_P$ on $F$ by $\sign_P\s:=\sqrt{\sign_P T_\s}$. (Qu\'eguiner \cite{Que} deals with involutions of the second kind.) It is now clear that the involution $\s$ is positive with respect to $\leq_P$  if and only if its signature with respect to $\leq_P$ is maximal.
\end{remark}

Procesi and Schacher also define a notion of \emph{positive elements} in $(A,\s)$, cf. \cite[\S V]{PS}. For greater clarity we have adapted their definitions as follows:

\begin{defi}\mbox{}
\begin{enumerate}[\rm (1)]
\item
An ordering $\leq_P$ of $F$ is called a {\it $\sigma$-ordering} if it makes the
involution $\sigma$ positive, i.e., if 
\[0\leq_P \Trd(\sigma(x)x) \quad \text{for all }x\in A.\]
\item
Suppose $\leq_P$ is a $\sigma$-ordering on $F$. An element $a\in \Sym(A,\s)$ is
called $\s$-\emph{positive with respect to $\leq_P$} if the quadratic form
$ \Trd(\sigma(x)ax)$
is positive semidefinite with respect to $\leq_P$. That is, if
\[0\leq_P\Trd(\sigma(x)ax)\quad \text{for all } x\in A.\]
\item
An element $a\in \Sym(A,\s)$ is called \emph{totally $\s$-positive} if it is positive
with respect to {all} $\sigma$-orderings on $F$.
\end{enumerate}
\end{defi}

Elements of the form $\s(x)x$ with $x\in A$ are called \emph{hermitian squares}. The set of hermitian squares of $A$ is clearly a subset of $\Sym(A,\s)$. It is also clear that the hermitian squares of $K$ are all in $F$. 

\begin{exam}\label{ex2.2}  Sums of hermitian squares and sums of traces of hermitian squares are
examples of totally $\s$-positive elements, as easy verifications will show.
\end{exam}

One of the main results in \cite{PS} explains that these are essentially
the only examples.
It can be considered as a noncommutative analogue of Artin's
solution to Hilbert's 17th problem:

\begin{theorem}{\cite[Theorem~5.4]{PS}}
\label{thm:NCh17}
Let $A$ be a central simple algebra with involution~$\s$, centre
$K$ and fixed field $F$.
Let $\alpha_1,\ldots,\alpha_m\in F$ be elements appearing in 
a diagonalization of the quadratic form $\Trd(\sigma(x)x)$. Then
for $a\in\Sym (A,\sigma)$ the following statements are equivalent:
\begin{enumerate}[\rm (i)]
\item
$a$ is totally $\s$-positive;
\item
there exist $x_{i,\ve}\in A$ with
$$
a=\sum_{\ve\in \{0,1\}^m} \alpha^\ve \sum_{i} \sigma(x_{i,\ve}) 
x_{i,\ve}.
$$
$($As usual, $\alpha^\ve$ denotes $\alpha_1^{\ve_1}\cdots \alpha_m^{\ve_m}$.$)$
\end{enumerate}
\end{theorem}

In the case $n=\deg A=2$, the weights $\alpha_j$ are superfluous (we will come back to this later).
Procesi and Schacher \cite[p.~404]{PS} conjecture that this is also the case for $n>2$:

\begin{conj}
\label{open:PS}
In a central simple algebra $A$ with involution $\s$,
every
totally $\s$-positive element is a sum of hermitian squares. \textup{(}Equivalently:
the trace of a
hermitian square is always a sum of hermitian squares.\textup{)}
\end{conj}

\begin{remark}\label{rem2.4} The two statements in the PS Conjecture are indeed equivalent: the necessary direction follows from the fact that traces of hermitian squares are totally $\s$-positive, as observed in Example~\ref{ex2.2}.

For the sufficient direction, assume that the trace of a
hermitian square is always a sum of hermitian squares.
Let $a\in \Sym(A,\s)$ be totally $\s$-positive. Then $a$ can be expressed in terms of the entries in a diagonalization of the form $\Trd(\s(x)x)$ as in Theorem~\ref{thm:NCh17}(ii). 
Let $\beta$ be such an entry. 
Thus, $\beta=\Trd(\s(y)y)$ for some $y\in A$. By the assumption
there are $x_1,\ldots, x_\ell\in A$ such that $\beta=\sum_i \s(x_i)x_i$. Since $\beta\in F$, the expression in Theorem~\ref{thm:NCh17}(ii) can now be rewritten as a sum of hermitian squares.
\end{remark}

As mentioned a few lines earlier, Procesi and Schacher provide supporting evidence for their conjecture for the case $\deg A=2$. Another  case where the PS~Conjecture is true has been well-known since the 1970s:

\begin{exam}\label{ex:psd}
Let $A$ be the full matrix ring $M_n(F)$ over a formally
real field
$F$ endowed with the transpose involution $\s=t$. 
Since $\Trd=\tr$, every ordering of $F$ is a $\sigma$-ordering. 
We claim that $a\in\Sym(A,\sigma)$ is totally $\s$-positive if and only if 
$a$ is a positive semidefinite matrix in $A\otimes_F R 
=M_n(R)$
for any real closed field $R$ containing $F$ (equivalently:
for any real closure of $F$). 

Indeed, if $a$ is totally $\s$-positive, then for all $x\in A$, 
$\tr(x^\T ax)$ is positive with respect to every ($\sigma$-)ordering
of $F$, i.e., $\tr(x^\T ax)\in\sos F$.
A diagonalization of the quadratic form $\tr(x^\T ax)$ will
 contain only sums of squares in $F$ (as it would otherwise violate
the total $\s$-positivity).
Hence this quadratic form remains positive semidefinite under every
ordered field extension of $F$.

The converse implication is also easy: if $a$ is positive semidefinite over
$M_n(R)$ for every real closed field $R\supseteq F$, then 
the trace of $x^\T ax$ for $x\in A$ 
is nonnegative under the ordering of $R$
and hence under all orderings of $F$. By definition, this means that
$a$ is totally $\s$-positive.

Moreover,
every totally $\s$-positive element of $(A,\sigma)$ 
is a sum of hermitian squares.
Essentially, this goes back to Gondard and Ribenboim \cite{GR}
and has been reproved several times \cite{Djo,FRS,HN,KS}.
It also follows easily from
Theorem~\ref{thm:NCh17} for it
suffices to show that the trace of a hermitian square is a sum of
hermitian squares. But this is clear: if $a=
\big[
a_{ij}
\big]
_{1\leq i,j\leq n}\in A$, then
\[\Trd(\sigma(a)a)=\sum_{i,j=1}^n a_{ij}^2\]
is obviously a sum of (hermitian) squares in $F$.

The reader will have no problems extending this example to the case
$K=F(\sqrt{-1})$ and $A=M_n(K)$ endowed with the conjugate transpose 
involution $\bbar t$.
\end{exam}

\section{The Counterexamples}
\label{sec:counter}

When the transpose involution in the previous example is replaced by an arbitrary orthogonal involution $\s$ on $M_n(F)$ (i.e., an involution which is adjoint to a quadratic form over $F$), the equivalence between totally $\s$-positive elements and sums of hermitian squares is in general no longer true, as we proceed to show in this section. We assume throughout that $F_0$ is a formally real field.

\begin{lemma}\label{lem3.1} Let $F=F_0(\!(X)\!)(\!(Y)\!)$, the iterated Laurent series field in two commuting variables $X$ and $Y$. The quadratic form
\[q=\<X,Y,XY\>\]
does not weakly represent $1$ over $F$. In fact this is already true over the rational function field $F_0(X,Y)$.
\end{lemma}

\begin{proof} Assume for the sake of contradiction that 
$m\x q$ represents $1$ for some positive integer $m$. 
Then the form
\[\vf:=\<1\>\perp m\x \<-X,-Y,-XY\>\]
is isotropic over $F$. 
This leads to a contradiction by repeated application of Springer's theorem on fields which are complete with respect to a discrete valuation, cf. \cite[Chapter VI, \S1]{Lam}. 
Since $F_0(X,Y)$ embeds into $F$ the proof is finished.
\end{proof}

\begin{theorem}\label{thm3.2} Let $F=F_0(X, Y)$. Let $A=M_3(F)$ and $\s=\ad_q$, where
\[q=\<X,Y,XY\>.\]
The \textup{(}$\s$-symmetric\textup{)} element $XY$ is totally $\s$-positive, but is not a sum of hermitian squares in $(A,\s)$.
\end{theorem}

\begin{proof} It is clear that $XY\in\Sym(A,\s)$ since $XY\in F$.

We first show that $XY$ is totally $\s$-positive. Since
$T_\s \simeq q\ox q$ (see \cite[p. 227]{Lew} or \cite[11.4]{BOI})
we have 
\[\sign_P T_\s=(\sign_P q)^2 \in \{1,9\}\] 
for any ordering $P\in X_F$. 
Hence, the set of $\s$-orderings on $F$ is not empty. It is exactly the set of $P\in X_F$ with $\sign_P T_\s=9$. (Note that $F$ has orderings for which both $X$ and $Y$, and thus $XY$, are positive so that the value $\sign_P T_\s=9$ can indeed be attained.)

Let $P$ be any $\s$-ordering on $F$. Then we have for any $a\in A$, 
\[\Trd(\s(a)a)\geq_P 0\] 
(by definition) and so for any $a\in A$,
\[\Trd\bigl(\s(a) XY a\bigr)=XY \Trd (\s(a)a)\geq_P 0,\]
since $XY\geq_P 0$ (for otherwise $\sign_P T_\s=1$ and $P$ would not be a $\s$-ordering on $F$).
Hence, $XY$ is totally $\s$-positive.
An alternative argument showing that
$XY$ is totally $\s$-positive can be given by observing that
$XY= \Trd(\s(b) b)$ for
$$b=\left[\begin{smallmatrix}
0&X&0\\
0&0&0\\
0&0&0
\end{smallmatrix}\right].
$$

Next we show that $XY$ is not a sum of hermitian squares in $(A,\s)=(M_3(F), \ad_q)$. We identify $XY$ with $XYI_3$ in $M_3(F)$, where $I_3$ denotes the $3\x 3$ identity matrix. Assume for the sake of contradiction that $XYI_3$ is a sum of elements of the form $\s(a)a$ with $a=[a_{ij}]_{1\leq i,j\leq 3} \in M_3(F)$. Recall that
\[\s(a)a =\ad_q(a)a=
\left[\begin{smallmatrix}
X & & \\
& Y & \\
& & XY
\end{smallmatrix}\right] \cdot a^t\cdot 
\left[\begin{smallmatrix}
X & & \\
& Y & \\
& & XY
\end{smallmatrix}\right]^{-1} \cdot a. \]
The $(3,3)$-entry of $\s(a)a$ is equal to
\[Ya_{13}^2+X a_{23}^2 + a_{33}^2.\]
By our assumption there are $s_1, s_2, s_3 \in \sum F^{\x 2}$ such that
\[XY=Y s_1+ Xs_2 + s_3,\]
which is equivalent with
\[1= X^{-1}s_1 + Y^{-1} s_2 + X^{-1} Y^{-1} s_3.\]
Thus, $1$ is weakly represented by the quadratic form
\[\<X^{-1}, Y^{-1}, X^{-1} Y^{-1}\> \simeq \<X, Y, XY\>=q,\]
which is impossible by Lemma~\ref{lem3.1}. This finishes the proof.
\end{proof}

The previous theorem gives us a counterexample to the PS~Conjecture. It shows that the conjecture is in general not true for full matrix algebras equipped with an orthogonal involution. In contrast, when we equip a full matrix algebra with a \emph{symplectic} involution, we will show in  Theorem~\ref{thm4.7} below that the conjecture does hold. 

Thus, we could ask if the PS~Conjecture also holds for non-split central simple algebras with symplectic involution. The answer is ``no'':

\begin{theorem}\label{thm3.3}  Let $F=F_0(X, Y)$. Let $A=M_3(F)\ox_F H\cong M_3(H)$,
where $H=(-1,-1)_F$ is Hamilton's quaternion division algebra over $F$. Equip $A$ with the involution $\s=\ad_q\ox\gamma$, where $\gamma$ is quaternion conjugation and $\s=\ad_q$ for
\[q=\<X,Y,XY\>.\]
The algebra $A$ is central simple over $F$ of degree $6$ and the involution $\s$ is symplectic.
The \textup{(}$\s$-symmetric\textup{)} element $XY$ is totally $\s$-positive, but is not a sum of hermitian squares in $(A,\s)$.
\end{theorem}

\begin{proof} The assertion about $(A,\s)$ is clear, as is the fact that $XY\in\Sym(A,\s)$ since $XY\in F$. 

It is easy to verify that the involution trace form of $\gamma$, $T_\gamma$, is isometric to $\<2\>\ox N_H$, where $N_H=\<1,1,1,1\>$ is the norm form of $H$. Here  $N_H(x):=\Nrd_H(x)$ for all $x\in H$, where $\Nrd_H$ denotes the reduced norm on $H$. 
Since $T_\s =T_{\ad_q\ox \gamma}\simeq T_{\ad_q}\ox T_\gamma$, we have 
\[\sign_P T_\s= (\sign_P T_{\ad_q})(\sign_P T_\gamma) = 4\sign_P T_{\ad_q}\in \{4, 36\}\]
for any ordering $P\in X_F$. Hence, the set of $\s$-orderings on $F$ is not empty. It is exactly the set of $P\in X_F$ with $\sign_P T_\s=36$. (Note again that this value can indeed be attained since there are orderings on $F$ for which both $X$ and $Y$, and thus $XY$, are positive.)
Arguing similarly as in the proof of Theorem~\ref{thm3.2} we can
verify that $XY$ is totally $\s$-positive.

Before proceeding, note that the involution $\gamma$ is adjoint to the hermitian form $\<1\>_\gamma$ over $(H,\gamma)$. Hence, $\s$ is adjoint to the hermitian form $h=q\ox \<1\>_\gamma =\<X,Y,XY\>_\gamma$ over $(H,\gamma)$. Thus
\[h(x,y)= \gamma(x_1) X y_1 +\gamma(x_2) Y y_2 + \gamma(x_3)XY y_3\]
for vectors $x=(x_1,x_2,x_3)$ and $y=(y_1,y_2,y_3)$ in the right $H$-vector space $H^3$.

Next we show that $XY$ is not a sum of hermitian squares in $(A,\s)=(M_3(H), \ad_h)$. We identify $XY$ with $XYI_3$ in $M_3(H)$, where $I_3$ denotes the $3\x 3$ identity matrix. Assume for the sake of contradiction that $XYI_3$ is a sum of elements of the form $\s(a)a$ with $a=[a_{ij}]_{1\leq i,j\leq 3} \in M_3(H)$. Recall that
\[\s(a)a =\ad_h(a)a=
\left[\begin{smallmatrix}
X & & \\
& Y & \\
& & XY
\end{smallmatrix}\right] \cdot \gamma(a)^t\cdot 
\left[\begin{smallmatrix}
X & & \\
& Y & \\
& & XY
\end{smallmatrix}\right]^{-1}\cdot a, \]
where $\gamma(a)=\big[\gamma(a_{ij})\big]_{1\leq i,j\leq 3}$.
The $(3,3)$-entry of $\s(a)a$ is equal to
\[\gamma(a_{13})Ya_{13}+\gamma(a_{23})X a_{23} +\gamma(a_{33}) a_{33}=YN_H(a_{13})+XN_H(a_{23})+N_H(a_{33}).\]
Since $N_H=\<1,1,1,1\>$, each of $N_H(a_{13})$, $N_H(a_{23})$, $N_H(a_{33})$ is a sum of four squares in $F$. Thus, by our assumption there are $s_1, s_2, s_3 \in \sum F^{\x 2}$ such that
\[XY=Y s_1+ Xs_2 + s_3.\]
We can now finish the proof with an appeal to Lemma~\ref{lem3.1}, as in the proof of Theorem~\ref{thm3.2}.
\end{proof}

\begin{remark} By tensoring $(M_3(F), \ad_q)$  with Hamilton's quaternion division algebra, equipped with a \emph{unitary} involution one obtains a counterexample in the non-split unitary case. We leave the details, which are similar to those in the proof of Theorem~\ref{thm3.3}, to the diligent reader.
\end{remark}

\begin{remark}
From a real algebra perspective it is clear that these counterexamples
to the PS~Conjecture can easily be seen to work over any formally real 
field $F$ that admits a proper semiordering (see 
\cite[\S 5]{PD} for details and unexplained terminology).
Given such a field $F$, endowed with a proper semiordering, take negative 
$a,b\in F$ such that $ab$ is negative as well.
Then $q=\langle a,b,ab\rangle$ does not weakly represent $1$ 
(the quadratic module generated by $\{-a,-b,-ab\}$ is proper) 
and thus
in $M_3(F)$, endowed with the involution $\sigma=\ad_q$, the element
$ab$ is totally $\sigma$-positive, but not a sum of hermitian squares
(as the proof of Theorem \ref{thm3.2} shows).
\end{remark}

\section{Positive Results}\label{para.pos.rel}

Procesi and Schacher \cite[p.~404 and 405]{PS} 
prove their conjecture for central simple algebras $A$ of degree two, i.e., quaternion algebras, with arbitrary involution $\s$ by 
appealing to matrices and the Cayley--Hamilton theorem.
We start this section by giving an alternative argument 
motivating some of the generalizations that follow.

Throughout this section we assume that the base field $F$ is formally real.

\begin{propo}\label{prop4.1} Let $A$ be a quaternion algebra \textup{(}not necessarily division\textup{)} with centre $K$, equipped with an arbitrary involution $\s$. Let $F$ be the fixed field of $(A,\s)$.  Each entry occurring in a diagonalization of $T_\s$ is a sum of hermitian squares.
\end{propo}

\begin{proof} (i) We first consider involutions of the first kind on $A$.
Let $A$ be the quaternion algebra $(a,b)_F$ with $F$-basis $\{1,i,j,k\}$ where $i$, $j$ and $k$ anti-commute, $ij=k$, $i^2=a$ and $j^2=b$.

If $\s$ is symplectic, then $\s$ is the unique quaternion conjugation involution $\gamma$ on $A$. An easy computation gives $T_\s=T_\gamma\simeq \<2\>\ox \<1,-a,-b,ab\>$.  We have
\[1=\gamma(1)1,\ -a=\gamma(i)i,\ -b=\gamma(j)j,\ ab=\gamma(k)k.\]

If $\s$ is orthogonal, then $\s=\Int(u)\circ\gamma$, where $u\in A$ satisfies $\gamma(u)=-u$. From \cite[11.6]{BOI} we know that
\[T_\s \simeq \<2\>\ox \<1, \Nrd_A(u), -\Nrd_A(s), -\Nrd_A(su)\>\]
for some $s\in A$ with $\s(s)=s=-\gamma(s)$. Now, 
\begin{align*}
\Nrd_A(u)&=u\gamma(u) =u\gamma(u) u^{-1}u=\s(u)u;\\
-\Nrd_A(s)&=-\gamma(s)s=\s(s)s;\\
-\Nrd_A(su)&=-\Nrd_A(s)\Nrd_A(u)=-\gamma(s)s\Nrd_A(u)=\s(s)\s(u)us=\s(us)us.
\end{align*}

(ii) Finally, let  $K=F(\sqrt{\delta})$ and let $A$ be a quaternion algebra over $K$ with unitary involution $\s$ whose restriction to $K$ is $\tau$, where $\tau$ is determined by $\tau(\sqrt{\delta})=-\sqrt{\delta}$. 
By a well-known result of Albert \cite[2.22]{BOI} there exists a unique quaternion $F$-subalgebra $A_0\subseteq A$  such that
\[A=A_0\ox_F K \text{ and } \sigma= \gamma_0\ox \tau,\]
where $\gamma_0$ is quaternion conjugation on $A_0$.  Then $T_\s\simeq T_{\gamma_0}\ox T_\tau\simeq T_{\gamma_0}\ox \<1,-\delta\>$. Since $\tau(\sqrt{\delta})\sqrt{\delta}=-\delta$, we are finished by the symplectic part of the proof.
\end{proof}

This shows in particular that the PS~Conjecture is true for full matrix algebras of degree two over a formally real field $F$ since these are just split quaternion algebras. 

Part (ii) of the proof of Proposition~\ref{prop4.1} motivates the following more general result:

\begin{theorem}\label{thm4.2} Let $A$ and $B$ be central simple algebras with centre $K$, equipped with arbitrary involutions $\s$ and $\tau$, respectively. Assume that $(A,\s)$ and $(B,\tau)$ have the same fixed field $F$. 
If the PS~Conjecture holds for $(A,\s)$ and $(B,\tau)$, it also holds for the tensor product $(A\ox_K B, \s\ox \tau)$.
\end{theorem}

\begin{proof} This is a simple computation, using the fact that $T_{\s\ox\tau}\simeq T_\s \ox T_\tau$ and that elements of $A$ commute with elements of $B$ in the tensor product $A\ox_K B$.
\end{proof}

\begin{coro}\label{cor4.3} Let $(Q_1,\s_1),\ldots, (Q_\ell,\s_\ell)$ be quaternion algebras with arbitrary involution over $K$ and with common fixed field $F$. The PS~Conjecture holds for the tensor product $\bigotimes_{i=1}^\ell (Q_i,\s_i)$. 
\end{coro}

\begin{proof} This is an  immediate consequence of Proposition~\ref{prop4.1} and Theorem~\ref{thm4.2}. 
\end{proof}

\begin{coro} Let $A=M_n(F)$ be a split algebra of $2$-power degree $n=2^\ell$, equipped with an orthogonal involution $\s$ which is adjoint to an $n$-fold Pfister form over $F$. The PS~Conjecture holds for $(A,\s)$.
\end{coro}

\begin{proof} By Becher's proof of the Pfister Factor Conjecture \cite{Bec}, $(A,\s)$ decomposes as
\[(A,\s)\cong \bigotimes_{i=1}^\ell (Q_i,\s_i),\]
where $(Q_1,\s_1),\ldots, (Q_\ell,\s_\ell)$ are quaternion algebras with involution. An appeal to Corollary~\ref{cor4.3} finishes the proof.
\end{proof}

\begin{coro}\label{cor4.4} Let $A=M_n(K)$ be a split algebra of $2$-power degree $n=2^\ell$, equipped with a hyperbolic involution $\s$ of any kind. Let $F$ be the fixed field of $(A,\s)$.  The PS~Conjecture holds for $(A,\s)$.
\end{coro}

\begin{proof} Recall from \cite[Theorem~2.1]{BST} that the involution $\s$ is hyperbolic if there exists an idempotent $e\in A$ such that $\s(e)=1-e$ or, equivalently, if the adjoint (quadratic, alternating or hermitian) form of $\s$ is hyperbolic. 

If $\ell=1$ this is just the split version of Proposition~\ref{prop4.1}. Assume now that $\ell\geq 2$.
By \cite[Theorem~2.2]{BST}, $(A,\s)$ decomposes as
\[(A,\s)\cong \bigotimes_{i=1}^\ell (Q,\s_i),\]
where $Q=M_2(K)$ and $\s_1,\ldots, \s_\ell$ are involutions on $Q$.  An appeal to Corollary~\ref{cor4.3} finishes the proof.
\end{proof}

\begin{coro} Let $A=M_n(F)$ be a split algebra of $2$-power degree $n=2^\ell$, equipped with a symplectic involution $\s$.  The PS~Conjecture holds for $(A,\s)$.
\end{coro}

\begin{proof} If $\s$ is a symplectic involution, it is hyperbolic (since it is adjoint to an alternating form over $F$ which is automatically hyperbolic) and we are finished by Corollary~\ref{cor4.4}.
\end{proof}

In fact, the PS~Conjecture is true for \emph{any} split algebra  with symplectic involution. Such an algebra is always of even degree.

\begin{theorem}\label{thm4.7} Let $A=M_n(F)$ be a split algebra of even degree $n=2m$, equipped with a symplectic involution $\s$.  The PS~Conjecture holds for $(A,\s)$.
\end{theorem}

\begin{proof} Since $\s$ is symplectic, the quadratic form $T_\s$ is hyperbolic (see \cite[p.~227]{Lew} or \cite[Proof of 11.7]{BOI}). Thus $T_\s \simeq m\x \<1,-1\>$ and it suffices to show that $-1$ is a sum of hermitian squares in $A$. We identify $-1$ with $-I_n$, where $I_n$ denotes the $n\x n$ identity matrix in $A=M_n(F)$.

Since $\s$ is symplectic, we have
$\s=\Int(S)\circ t$, where $t$ denotes transposition and $S\in \GL_n(F)$ satisfies $S^t=-S$. 
Since $S$ is skew-symmetric, there exists a matrix $P\in \GL_n(F)$ such that $P^tSP=B$, where $B$ is the block diagonal matrix with $m$ blocks 
$\left[\begin{smallmatrix}
0&1\\
-1 & 0
\end{smallmatrix}\right]$ on the diagonal. 

Let $X$ be the block diagonal matrix with $m$ blocks
$\left[\begin{smallmatrix}
0&1\\
1 & 0
\end{smallmatrix}\right]$ on the diagonal. Then $X^tBX=B^{-1}$. 
Hence with $Y=PXP^t$, we have $Y^tSY=S^{-1}$.
Thus
\[
\s(SY)SY=S (SY)^t S^{-1}SY= SY^tS^tY=SY^t(-S)Y=-SS^{-1}=-I_n.
\qedhere
\]
\end{proof}

\section{Positive noncommutative polynomials}\label{sec5}

\subsection{Algebras of generic matrices with involution}
\label{sec:generic}

After studying the PS~Conjecture in the setting of
central simple algebras with involution, we proceed to interpret
these results as well as Theorem \ref{thm:NCh17} for
non-dimensionfree positivity of noncommutative (NC) polynomials.

Motivated by problems in optimization and control theory,
Helton \cite{Hel} 
proved that a symmetric real or complex
NC polynomial, all of  
whose images under algebra $*$-homomorphisms into 
$M_n(\R)$, $n\in\N$, 
are positive semidefinite (i.e., a dimensionfree positive NC polynomial), 
is a sum of hermitian squares. What we are interested in, is 
positivity under evaluations in $M_n(\R)$ for a \emph{fixed} $n$.

To tackle this problem we 
introduce the language of generic matrices, 
cf.~\cite[Chapters 1 and 3]{Pro1} or \cite[\S 1.3]{Row}.
Verifying a condition on evaluations of an NC polynomial
in the algebra of $n\times n$ matrices is often conveniently done
in the algebra of generic matrices. 
 In this subsection we recall
the definition of generic matrices with involution, 
while our main result on positive NC polynomials
(i.e., a Positivstellensatz) is presented in the next subsection.

As in the classical construction of the algebra of generic
matrices,
it is possible to construct the algebra of generic matrices \emph{with
involution}, see e.g.~\cite[\S 20]{Pro2} or \cite[\S II]{PS}. 
To each type of involution (orthogonal, symplectic
and unitary) an algebra of generic matrices
with involution can be associated, as we now explain.
We assume from now on that 
$K$ is a field of characteristic
$0$ with involution $*$ and fixed field $F$.

Let $K\axstar$ be the free algebra with
involution over $(K,*)$, i.e.,  the algebra with involution,
freely generated by the
noncommuting variables $\bar X:=
(X_1,X_2,\ldots )$. Its elements (called \emph{NC polynomials})
are (finite) linear combinations of
words in (the infinitely many) letters $\bar X, \bar X^*$.

Fix a type J $\in\{$orthogonal, symplectic, unitary$\}$.
Let $\got a_{\J_n}\subseteq K\axstar$ 
denote the ideal of all identities satisfied by degree $n$ central simple
$K$-algebras with type $\J$ involution.
That is, $f=f(X_1,\ldots,X_k,X_1^*,\ldots
,X_k^*)\in K\axstar$ is an element of $\got a_{\J_n}$ if and only if for every
central simple algebra $A$ of degree $n$ with type $\J$ involution $\sigma$  
and every $a_1,\ldots,a_k\in A$, 
\[f(a_1,\ldots,a_k,\sigma(a_1),\ldots,\sigma(a_k))=0.\]
Then $\GM_n(K,{\J}):=
K\axstar/\got a_{\J_n}$ is the \emph{algebra of generic $n\times n$ 
matrices
with type $\J$ involution}.

\begin{remark} An alternative description of the algebra of generic matrices with involution can be obtained as follows.
Let $\zeta:=(\zeta_{ij}^{(\ell)}\mid1\leq i, j\leq n,\,\ell\in\N)$ 
denote commuting
variables and form the polynomial algebra $K[\zeta]$ endowed
with the involution extending $*$ and fixing $\zeta_{ij}^{(\ell)}$
pointwise.
Consider the
$n\times n$ matrices $Y_{\ell}:=\big[
\zeta_{ij}^{(\ell)}\big]_{1\leq i,j\leq n}\in M_n(K[\zeta])$,  
$\ell\in\N$. Each $Y_{\ell}$ is
called a \textit{generic matrix}. 

\begin{enumerate}[\rm (a)]
\item
If  $\J\in\{$orthogonal, unitary$\}$, then 
the (unital) $K$-subalgebra of $M_n(K[\zeta])$ 
generated by the $Y_\ell$ and their
transposes is (canonically) isomorphic to $\GM_n(K,{\rm J})$.
\item
If $\J=$ symplectic, then $n$ is even, say $n=2m$. Consider
the usual symplectic involution 
$$\begin{bmatrix}
x&y\\
z&w
\end{bmatrix}
\mapsto
\begin{bmatrix}
w^t&-y^t\\
-z^t&x^t
\end{bmatrix}
$$
on $M_{2m}(K[\zeta])$. Then the (unital) $K$-subalgebra of $M_n(K[\zeta])$ 
generated by the $Y_\ell$ and their
images under this involution 
is (canonically) isomorphic to $\GM_n(K,{\J})$.
\end{enumerate}
\end{remark}

If $n=1$, then $\J\in\{$orthogonal, unitary$\}$ and 
$\GM_1(K,\J)$ is isomorphic to $K[\zeta]$ endowed with the involution
introduced above. Hence in the sequel we will always assume $n\geq 2$.

Let $\J\in\{$orthogonal, symplectic, unitary$\}$.
For $n\geq 2$, $\GM_n(K,\J)$ is a 
PI algebra and a domain (cf.~\cite[\S II]{PS}). Hence its
central localization is a division algebra $\UD_n(K,\J)$ with involution, 
which we call
the \emph{universal division algebra} with type $\J$  involution of degree~$n$.
As we will only consider the canonical involution on $\GM_n(K,\J)$ and
$\UD_n(K,\J)$ we use $*$ to denote it.

\begin{remark}
Our approach to generic matrices is purely algebraic. A
representation-theoretic viewpoint with a more geometric flavour 
can be found in \cite{Pro2}.
\end{remark}

\subsection{A Positivstellensatz}\label{sec6}

Let $\KK\in\{\R,\C\}$ be endowed with the complex conjugation involution $\bbar$.
Our aim in this subsection is to deduce a non-dimensionfree
version of Helton's sum of hermitian squares theorem.
We will describe symmetric NC polynomials all of whose evaluations in 
$M_n(\KK)$ are positive semidefinite, see Theorem~\ref{thm:positiv}.

The main line of reasoning is the same as in \cite[\S 4]{PS},
while the dependence on Tarski's transfer principle from real
algebraic geometry is isolated and emphasized in Lemma \ref{lem:uf} below.
The lemma characterizes 
total $*$-positivity in the algebra of
generic matrices $\GM_n(\KK,\J)$. 
Its proof uses some elementary
model theory, e.g.~Tarski's transfer principle for real closed fields. 
All the necessary 
background can be found in 
\cite[\S 1 and \S 2]{PD} or, alternatively, \cite[\S 1]{BCR}.

\begin{lemma}\label{lem:uf}
Let $n\in \N$. If $\KK=\R$, let $\J=$ orthogonal and if
$\KK=\C$, let $\J=$ unitary. 
If
$a=a^*\in\GM_n(\KK,\J)$ is totally 
$\sigma$-positive under each $*$-homomorphism from
$\GM_n(\KK,\J)$ to $M_n(\KK)$ endowed with a positive type $\J$ involution $\sigma$, then
$a$ is totally $*$-positive $($in $\UD_n(\KK,\J))$.
\end{lemma}

\begin{proof}
Suppose $a\in \GM_n(\KK,\J)$ is not totally $*$-positive. Then there 
is a $*$-ordering $\leq$ of the fixed field $Z$ of the centre 
of $\UD_n(\KK,\J))$ under which
$\Trd(x^*ax)$ is not positive semidefinite. 
Let $\langle \alpha_1,\ldots, \alpha_{m}\rangle$
be the diagonalization
of $\Trd(x^*x)$ with $\alpha_i=\alpha_i^* \in Z$. 
(Here $m=n^2$ if the involution is of the first kind and $m=2n^2$
otherwise.)
Given that $Z$ is the field of
fractions of the symmetric centre $Z_0$ of $\GM_n(\KK,\J)$, we may even assume $\alpha_i\in
Z_0$. We also diagonalize $\Trd(x^*ax)$ as $\langle \beta_1,\ldots, \beta_{m}\rangle$ with $\beta_i\in Z_0$. Clearly, $\alpha_i>0$ and one of the $\beta_i$,
say $\beta_1$, is negative with respect to the given $*$-ordering $\leq$. 
Let $\bar Z^\rc$ denote the real closure of $Z$ with respect to this
ordering and form $A:=\UD_n(\KK,\J)\otimes_Z \bar Z^\rc$ endowed
with the involution $\sigma=*\otimes {\rm id}$. Then $A$ is
a central simple algebra over a real closed (if $\J=$ orthogonal)
or algebraically closed field (if $\J=$ unitary). Moreover, its involution
$\sigma$
is positive. Hence by the classification result \cite[Theorem 1.2]{PS} of Procesi and 
Schacher, $A$ is either $M_n(\bar Z^\rc)$ 
endowed with the transpose (if $\J=$ orthogonal) or
$M_n(\bar Z)$ endowed with the complex conjugate transpose involution (if $\J=$ unitary).
Here $\bar Z$ is the algebraic closure $\bar Z^\rc(\sqrt{-1})$ of 
$\bar Z^\rc$ and the complex conjugate maps $r+t \sqrt{-1} 
\mapsto r- t \sqrt{-1}$ for $r,t\in\bar Z^\rc$.

For $b\in\GM_n(\KK,\J)$ let $\hat b\in \KK\axstar$ denote a
preimage of $b$ under the canonical map $\KK\axstar\to\GM_n(\KK,\J)$.
Every $*$-homomorphism $\GM_n(\KK,\J)\to M_n(L)$ for a $*$-field extension
$L$ of $\KK$, where $M_n(L)$ is given a type $\J$ involution,
yields a $*$-homomorphism $\KK\axstar\to M_n(L)$, so is
essentially given by a point $s\in M_n(L)^\N$ describing the images
of the $X_i$ under this induced map.

By construction, the image $\beta_1\otimes 1$ of $\beta_1$ under the
embedding of algebras with involution $\GM_n(\KK,\J)\to A$
is not $\sigma$-positive. Let $s$ denote the corresponding evaluation
point.
By Example \ref{ex:psd}, this means that $
\hat \beta_1(s,\bar{s}^t)=\beta_1\otimes 1$ is
not positive semidefinite. Consider the following elementary statement:
\begin{equation}\label{eq:1st}
\begin{split}
\exists \,n\times n \text{ matrices }x=(x_1,\ldots, x_N) : \; &
\hat\alpha_i(x,\bar{x}^t) \text{ is positive semidefinite } \land \\
& \hat\beta_1(x,\bar{x}^t) \text{ is not positive semidefinite}.
\end{split}
\end{equation}
($N$ is the maximal number of variables appearing
in one of the $\hat\alpha_i, \hat\beta_1$.)

Obviously such $n\times n$ matrices $x_i$ can be found over 
$\bar Z^\rc$ or $\bar Z$; just take $x_i=s_i$.
By Tarski's transfer principle, the above elementary statement \eqref{eq:1st}
can be satisfied in $\KK$. This yields a $*$-homomorphism
$\KK\axstar\to M_n(\KK)$ endowed with the (positive) involution $\bbar t$ and in
turn (by universality) a $*$-homomorphism $\GM_n(\KK,\J)\to (M_n(\KK),\bbar t)$. 
By the construction, the image of $a$ under this mapping will
not be positive semidefinite. 
This finishes the proof.
\end{proof}

In order to state the Positivstellensatz, we need to recall
the notion of {\it central polynomials} for $n\times n$ matrices. These are
$f\in K\axstar$ whose image in $\GM_n(K,\J)$ is central.
Equivalently, the image of $f$ under a $*$-homomorphism
from $K\axstar$ to $M_n(K)$ endowed with a type $\J$ involution,
is always a scalar matrix. 
If it is nonzero, we call $f$ {\it nonvanishing}.
The existence of nonvanishing central polynomials
is nontrivial; we refer to \cite[\S 1; Appendix A]{Row} for details.

\begin{theorem}[Positivstellensatz]\label{thm:positiv}
Suppose $\KK\in\{\R,\C\}$ is endowed with the complex conjugate involution $\bbar$.
Let $g=g^*\in \KK\axstar$, $n\in \N$ and 
fix a type
$\J\in\{$orthogonal, unitary$\}$ according to the type of involution
on $\KK$.
Choose $\alpha_1,\ldots,\alpha_m\in \KK\axstar$ whose images
in $\GM_n(\KK,\J)$ form a diagonalization of the quadratic
form $\Trd(x^*x)$ on $\UD_n(\KK,\J)$.
Then the following are equivalent:
\begin{enumerate}[\rm (i)]
\item
for any $s\in M_n(\KK)^\N$, $g(s,\bar{s}^t)$ is positive semidefinite;
\item 
there exists a nonvanishing central polynomial $h\in \KK\axstar$ for
$n\times n$ matrices
and $p_{i,\ve}\in \KK\axstar$
with
\[h^* g h \equiv
\sum_{\ve\in \{0,1\}^m} \alpha^\ve \sum_{i} p_{i,\ve}^*
p_{i,\ve}  \pmod {\got a_{\J_n}}.\]
\end{enumerate}
\end{theorem}

\begin{proof}
Given a congruence as in (ii), it is clear
that (i) holds whenever $h(s,\bar{s}^t)\neq 0$. As the set of all such $s$ is
Zariski dense, (i) holds for all $s\in M_n(\KK)^\N$.

For the converse implication note that 
by Lemma \ref{lem:uf}, $g+\got a_{\J_n}$ is totally $*$-positive
in $\UD_n(\KK,\J)$. Hence
by Theorem \ref{thm:NCh17} we obtain a positivity certificate
$$
g+\got a_{\J_n} =\sum_{\ve\in \{0,1\}^m} (\alpha+\got a_{\J_n})^\ve \sum_{i} (x'_{i,\ve})^* 
x'_{i,\ve}
$$
for some $x'_{i,\ve}\in\UD_n(\KK,\J)$.
Clearing denominators, there are $x_{i,\ve}\in\GM_n(\KK,\J)$ and
a nonzero central $r\in\GM_n(\KK,\J)$ with
$$r^*(g+\got a_{\J_n})r= \sum_{\ve\in \{0,1\}^m} (\alpha+\got a_{\J_n})^\ve \sum_{i} x_{i,\ve}^*
x_{i,\ve}.
$$
Lifting this equality to the free algebra yields the desired conclusion.
\end{proof}

When $n=2$, the weights $\alpha$ are redundant (cf.~\S\ref{para.pos.rel}
or \cite[p.~405]{PS}) and
we obtain the following strengthening:

\begin{coro}\label{cor:positiv2}
Suppose $\KK\in\{\R,\C\}$ is endowed with the complex conjugate involution~$\bbar$.
Let $g=g^*\in \KK\axstar$, $n\in \N$ and 
fix a type
$\J\in\{$orthogonal, unitary$\}$ according to the type of involution
on $\KK$.
Then the following are equivalent:
\begin{enumerate}[\rm (i)]
\item
for any $s\in M_2(\KK)^\N$, $g(s,\bar{s}^t)$ is positive semidefinite;
\item
there exists a nonvanishing central polynomial $h\in \KK\axstar$ for
$2\times 2$ matrices
and $p_{i}\in \KK\axstar$
with
\[
h^* g h \equiv
\sum_{i} p_{i}^*
p_{i}  \pmod {\got a_{\J_2}}.\]
\end{enumerate}
\end{coro}

\begin{remark}
By Tarski's transfer principle,
Theorem \ref{thm:positiv} and Corollary
\ref{cor:positiv2} hold with $\KK$ replaced by
any real closed  
or algebraically closed field of characteristic $0$.
\end{remark}

We conclude the paper with a problem:
can Theorem 
\ref{thm:positiv} be used to give a proof of Helton's sum of
hermitian
squares theorem?

\section*{Acknowledgments}{\small

Both authors wish to thank David Lewis for some very useful 
observations.
The second author wishes to thank Jaka Cimpri\v c for arranging a very nice stay at the University of Ljubljana in May 2008. The first author was supported by the Slovenian Research Agency (project no. Z1-9570-0101-06). The second author was supported by the Science Foundation Ireland Research Frontiers Programme (project no.
07/RFP/MATF191).}


\begin{thebibliography}{dOHMP}

\bibitem[BST]{BST} E. Bayer-Fluckiger, D.B. Shapiro, J.-P. Tignol, Hyperbolic involutions, \emph{Math. Z.}  \textbf{214}  (1993),  no. 3, 461--476.

\bibitem[Bec]{Bec}
K.J. Becher, A proof of the Pfister factor conjecture,  \emph{Invent. Math.}  \textbf{173}  (2008),  no. 1, 1--6.

\bibitem[BCR]{BCR}
J. Bochnak, M. Coste, M.-F. Roy, \emph{Real algebraic geometry},
Springer-Verlag, Berlin, 1998.

\bibitem[dOHMP]{dOHMP}
M.C. de Olivera, J.W. Helton, S.A. McCullough,
M. Putinar, Engineering Systems and Free Semi-Algebraic Geometry,
\emph{Emerging Applications of Algebraic Geometry}, 17--62,
IMA Vol. Math. Appl. \textbf{149}, Springer, 2008.


\bibitem[Djo]{Djo} D.\v Z. Djokovi\'c,
Positive semi-definite matrices as sums of squares,
\emph{Linear Algebra and Appl.} {\bf 14} (1976), no.~1, 37--40.

\bibitem[FRS]{FRS} J.F. Fernando, J.M. Ruiz, C. Scheiderer,
Sums of squares of linear forms,
\emph{Math. Res. Lett.} {\bf 13} (2006), no.~5-6, 947--956.

\bibitem[GR]{GR} D. Gondard, P. Ribenboim,
Le 17e probl\`eme de Hilbert pour les matrices,
\emph{Bull. Sci. Math.} (2) {\bf 98} (1974), no.~1, 49--56.


\bibitem[Hel]{Hel}
J.W. Helton, 
``Positive'' noncommutative polynomials are sums of squares,
\emph{Ann. of Math.} (2) {\bf 156} (2002), no.~2, 675--694. 

\bibitem[HN]{HN} C.J. Hillar, J. Nie,
An elementary and constructive solution to Hilbert's 17th problem for matrices,
\emph{Proc. Amer. Math. Soc.} {\bf 136} (2008), no.~1, 73--76.


\bibitem[KS]{KS} I. Klep, M. Schweighofer,
Positivity certificates for matrix polynomials,
in preparation.

\bibitem[KMRT]{BOI} M.-A. Knus, A.S. Merkurjev, M. Rost, J.-P. Tignol,
\emph{The Book of Involutions}, Coll. Pub.~{\bf 44}, Amer. Math. Soc.,
Providence, RI (1998).

\bibitem[Lam]{Lam} T.Y. Lam,  \emph{Introduction to quadratic forms over fields}, Graduate Studies in Mathematics  {\bf 67},  American Mathematical Society,  Providence,  RI  (2005).

\bibitem[Lew]{Lew} D.W. Lewis, Trace forms, Kronecker sums, and the shuffle matrix,
\emph{Linear and Multilinear Algebra} \textbf{40} (1996), no. 3, 221--227. 

\bibitem[LT]{LT} D.W. Lewis, J.-P. Tignol, On the signature of an
involution, \emph{Arch. Math.} {\bf 60} (1993), no.~2, 128--135.


\bibitem[PD]{PD}
A. Prestel, C.N. Delzell, \emph{Positive polynomials. From Hilbert's 17th problem to real algebra}, Springer-Verlag, Berlin, 2001.

\bibitem[Pro1]{Pro1}
C. Procesi, \emph{Rings with polynomial identities}, 
Marcel Dekker, Inc., New York, 1973.

\bibitem[Pro2]{Pro2}
C. Procesi, The invariant theory of $n\times n$ matrices, \emph{Advances in  
Math.} {\bf 19} (1976), no. 3, 306--381.


\bibitem[PS]{PS} C. Procesi, M. Schacher,
A non-commutative real Nullstellensatz and Hilbert's 17th problem,
\emph{Ann. of Math.} (2) {\bf 104} (1976), no.~3, 395--406.

\bibitem[Que]{Que} A. Qu\'eguiner, Signature des involutions de deuxi\`eme
esp\`ece, \emph{Arch. Math.} {\bf 65} (1995),  no.~5, 408--412.

\bibitem[Row]{Row}
L.H. Rowen, \emph{Polynomial identities in ring theory}, 
Pure and Applied Mathematics, {\bf 84}, Academic Press, Inc., 
New York-London (1980).

\bibitem[Wei]{Wei} A. Weil, 
Algebras with involutions and the classical groups,
\emph{J. Indian Math. Soc. (N.S.)}  \textbf{24} (1961), 589--623. 

\end{thebibliography}
\end{document}